\documentclass[12pt,leqno]{article}

\usepackage{amsmath,amssymb,amsthm}
\usepackage[all]{xy}
\usepackage{lscape}

\makeatletter

\newtheorem*{Theorem}{Theorem~\ref{main.thm}}

\newtheorem{theorem}{Theorem}[section]
\newtheorem{proposition}[theorem]{Proposition}
\newtheorem{lemma}[theorem]{Lemma}
\newtheorem{corollary}[theorem]{Corollary}

\theoremstyle{definition}

\theoremstyle{remark}

\theoremstyle{remark}

\def\({{\rm (}}
\def\){{\rm )}}

\let\Mathrm\operator@font

\newcommand{\fm}{\ensuremath{\mathfrak m}}
\newcommand{\fn}{\ensuremath{\mathfrak n}}
\newcommand{\fp}{\ensuremath{\mathfrak p}}

\def\standop#1{\mathop{\Mathrm #1}\nolimits}
\def\difstop#1#2{\expandafter\def\csname #1\endcsname{\standop{#2}}}
\def\defstop#1{\difstop{#1}{#1}}

\defstop{AB}
\defstop{ann}
\defstop{Ass}
\defstop{Add}
\defstop{Alt}
\defstop{Ass}
\difstop{adim}{AFHdim}

\defstop{Bl}

\defstop{CMFI}
\defstop{codim}
\defstop{Coh}
\defstop{coht}
\defstop{Coker}
\defstop{Cone}
\defstop{Cl}
\defstop{Cox}
\defstop{cosk}
\defstop{cd}
\defstop{cmd}
\difstop{charac}{char}

\defstop{depth}
\defstop{Der}
\defstop{Div}
\defstop{div}

\defstop{EM}
\defstop{embdim}
\defstop{End}
\defstop{ev}
\defstop{Ext}

\difstop{fdim}{flat.dim}
\defstop{Flat}
\defstop{Func}
\defstop{Fpqc}

\defstop{GD}
\defstop{Good}
\defstop{Gal}
\defstop{Grass}

\difstop{height}{ht}
\defstop{Hom}
\defstop{Hy}

\difstop{Image}{Im}
\defstop{ind}
\defstop{ini}

\defstop{Ker}

\defstop{Lch}
\defstop{length}
\defstop{Lin}
\defstop{Lqc}
\defstop{lqc}
\defstop{LQ}
\defstop{LM}
\defstop{Lie}
\defstop{Loc}

\defstop{Mat}
\defstop{Max}
\defstop{Min}
\defstop{Mod}
\defstop{Mor}
\defstop{MCM}
\defstop{Map}
\difstop{mred}{red}

\defstop{Nerve}
\defstop{NonSFR}
\defstop{NonCMFI}
\defstop{NonNor}
\defstop{Nor}

\defstop{ob}
\defstop{Ob}

\defstop{PA}
\defstop{PM}
\defstop{PR}
\defstop{Proj}
\defstop{Prin}
\defstop{Pic}

\defstop{Qch}
\defstop{qch}

\defstop{rad}
\defstop{rank}

\defstop{res}
\defstop{Reg}
\defstop{Ref}
\defstop{Rat}

\defstop{Spec}
\defstop{supp}
\defstop{Supp}
\defstop{Sym}
\defstop{Sing}
\defstop{SFR}
\defstop{Soc}
\defstop{Sh}

\difstop{tdeg}{trans.deg}

\defstop{Tor}
\defstop{TRD}
\difstop{trace}{tr}
\defstop{tot}

\defstop{Zar}
\defstop{FL}
\defstop{add}
\defstop{Ind}
\defstop{grmod}

\difstop{summ}{sum}

\def\({{\rm(}}
\def\){{\rm)}}

\def\fm{\mathfrak{m}}
\def\fp{\mathfrak{p}}

\def\ZZ{{\mathbb{Z}}}

\def\sdarrow#1{\downarrow\hbox to 0pt{\scriptsize$#1$\hss}}
\def\suarrow#1{\uparrow\hbox to 0pt{\scriptsize$#1$\hss}}
\def\ssearrow#1{\searrow\hbox to 0pt{\scriptsize$#1$\hss}}

\def\section{\@startsection{section}{1}{\z@ }%
  {-3.5ex plus -1ex minus -.2ex}{2.3ex plus .2ex}{\bf }}

\long\def\refname{\par\kern -3ex
  \begin{center}\rm R\sc{eferences}\end{center}\par\kern 
  -2ex}

\def\@seccntformat#1{\csname the#1\endcsname.\quad}

\def\@@@sect#1#2#3#4#5#6[#7]#8{%
  \ifnum #2>\c@secnumdepth 
  \def \@svsec {}\else \refstepcounter {#1}%
  \def\@svsec{}
  \fi 
  \@tempskipa #5\relax 
  \ifdim \@tempskipa >\z@ 
  \begingroup #6\relax \@hangfrom {\hskip #3\relax 
    \@svsec}{\interlinepenalty \@M #8\par }\endgroup 
  \csname #1mark\endcsname {#7}
  \else 
  \def \@svsechd {#6\hskip #3\@svsec #8\csname #1mark\endcsname {#7}}
  \fi \@xsect {#5}}

\def\@@@startsection#1#2#3#4#5#6{%
  \if@noskipsec \leavevmode \fi \par \@tempskipa #4\relax \@afterindenttrue 
  \ifdim \@tempskipa <\z@ \@tempskipa -\@tempskipa \@afterindentfalse 
  \fi \if@nobreak \everypar {}\else \addpenalty {\@secpenalty }\addvspace 
  {\@tempskipa }\fi \@ifstar {\@ssect {#3}{#4}{#5}{#6}}{\@dblarg 
    {\@@@sect {#1}{#2}{#3}{#4}{#5}{#6}}}}

\def\theparagraph{\thesection.\arabic{paragraph}}
\def\aparagraph{\@@@startsection{paragraph}{2}{\z@ }%
  {1.75ex plus .2ex minus .15ex}{-1em}{\bf(\theparagraph) } }
\def\paragraph{\@@@startsection{paragraph}{2}{\z@ }%
  {1.75ex plus .2ex minus .15ex}{-1em}{}{\bf(\theparagraph)} }

\let\c@theorem\c@paragraph

\advance\textheight 20mm
\advance\textwidth 16mm
\title{Hochster--Eagon type theorem\\
  for Serre's $(S_n)$ condition}
\author{Mitsuyasu H{\sc ashimoto}\thanks{Partially supported by JSPS KAKENHI Grant number 20K03538
    and MEXT Promotion of Distinctive Joint Research Center Program JPMXP0619217849.}}
\date{}

\begin{document}

\maketitle

\begin{abstract}
  Let $(A,\fm)\rightarrow (B,\fn)$ be a pure homomorphism between Noetherian commutative rings.
  If $B/\fm B$ is an Artinian ring, then we have $\dim A=\dim B$ and $\depth A\geq \depth B$.
  Using this version of Hochster--Eagon theorem,
  we prove the following.
  Let $A\rightarrow B$ be a pure homomorphism between Noetherian commutative
  rings. Assume that the fiber ring $\kappa(\fp)\otimes_A B$ is Artinian for each $\fp\in\Spec A$,
  and $B$ satisfies Serre's $(S_n)$ condition.
  Then $A$ also satisfies Serre's $(S_n)$ condition.
  In particular, if a finite group $G$ acts on $B$ and the order $|G|$ of $G$ is invertible in $B$,
  and if $B$ is Noetherian with the $(S_n)$ condition, then the ring of invariants $A=B^G$
  also satisfies the $(S_n)$ condition.
\end{abstract}

\section{Introduction}

We say that a homomorphism of commutative rings $f:A\rightarrow B$ is {\em pure} if for any
$A$-module $W$, the map $j_W:W\rightarrow B\otimes_AW$ given by $j_W(w)=1\otimes w$ is injective.
If $f:A\rightarrow B$ is pure, many good ring-theoretic properties of $B$ are inherited by $A$.
For example, if $B$ is Noetherian (resp.\ a normal domain), then so is $A$ \cite[(6.15)]{HR}.
The Hochster--Eagon theorem \cite[Proposition~12]{HE}, a striking result in this direction
appeared in 1971, states that if $f: A\rightarrow B$ is an integral homomorphism between commutative
rings, $B$ is Noetherian and Cohen--Macaulay, and $A$ is a direct summand subring of $B$ through $f$
(that is, there is a left inverse $g:B\rightarrow A$ as
an $A$-linear map (that is, $gf=1_A$)), then $A$ is also Noetherian and Cohen--Macaulay.
As an application, Hochster and Eagon proved that the non-modular invariant subring of a finite
group action on a Cohen--Macaulay ring is again Cohen--Macaulay \cite{HE}.
The assumption of integral extension is necessary, as any finitely generated domain over a field,
Cohen--Macaulay or not, is a direct summand subring of a finitely generated Cohen--Macaulay domain
by Kawasaki's arithmetic Macaulayfication theorem \cite[Theorem~1.3]{Kawasaki}.
Although Cohen--Macaulay property is not inherited by a pure subring in genral,
some important results in this direction are known, see \cite{Boutot,HH,HR,Schoutens}.

In this paper, we prove an analog of Hochster--Eagon theorem for Serre's $(S_n)$ condition.
We say that a Noetherian commutative ring $R$ satisfies the condition $(S_n)$ if
$\depth R_\fp\geq \min(n,\dim R_\fp)$ for any $\fp\in\Spec R$.
This is equivalent to say that for any $\fp\in\Spec R$ such that $\depth R_\fp<n$, we have $R_\fp$ is Cohen--Macaulay.
Our main theorem is the following.

\begin{Theorem}
  Let $f:A\rightarrow B$ be a pure ring homomorphism.
  Let $B$ be a Noetherian ring which satisfies Serre's $(S_n)$ condition.
  If the fiber ring $B\otimes_A \kappa(\fp)$ is zero-dimensional for any prime
  ideal $\fp$ of $A$, then $A$ satisfies $(S_n)$, too.
\end{Theorem}

If $f:(A,\fm)\rightarrow (B,\fn)$ is a pure local homomorphism such that $B/\fm B$ is Artinian,
then we have that $\dim A=\dim B$ and $\depth A\geq \depth B$ (Proposition~\ref{depth.prop}),
and the proof of Theorem~\ref{main.thm} is reduced to this version of Hochster--Eagon theorem.
As a corollary, we have a result in non-modular invariant theory of finite groups;
if a finite group $G$ acts on a Noetherian ring $B$ and
the order $|G|$ of $G$ is invertible in $B$, the $(S_n)$ property of $B$ is
inherited by $A$ for any $n\geq 1$ (Corollary~\ref{main.cor}).

\section{The results}

\paragraph
Let $R$ be a commutative ring and $f:M\rightarrow N$ be an $R$-linear map between $R$-modules.
We say that $f$ is {\em pure} or $R$-pure if for any $R$-module $W$, the map
$f\otimes 1_W: M\otimes_R W\rightarrow N\otimes_R W$ is injective.
Clearly, a pure linear map is injective.
A ring homomorphism between commutative rings $\varphi:A\rightarrow B$ is said to be pure
if it is so as an $A$-linear map.
For each $A$-module $W$, we denote by $j_W:W=A\otimes_AW\rightarrow B\otimes_A W$
the map $\varphi\otimes 1_W$.
So $\varphi$ is pure if and only if $j_W$ is injective for any $A$-module $W$.
If $A$ is a subring of $B$ and the inclusion map $A\hookrightarrow B$ is pure,
then we say that $A$ is a pure subring of $B$.

\paragraph
Let $\varphi:A\rightarrow B$ be a pure ring homomorphism, and $\fp \in\Spec A$.
Then $\kappa(\fp)=A\otimes_A\kappa(\fp)\rightarrow B\otimes_A\kappa(\fp)$ is injective, and hence
the fiber ring $B\otimes_A\kappa(\fp)$ is not zero, where $\kappa(\fp)= A_\fp/\fp A_\fp$ is the
residue field of $A_\fp$.
It follows that ${}^a\varphi:\Spec B\rightarrow \Spec A$ is surjective,
since $({}^a\varphi)^{-1}(\fp)\cong \Spec(B\otimes_A\kappa(\fp))\neq\emptyset$ for each $\fp$.
In particular, a pure ring homomorphism between local rings is a local homomorphism.

\paragraph
Let $\varphi:A\rightarrow B$ be a pure homomorphism.
If $B$ is Noetherian, then so is $A$ \cite[(6.15)]{HR}.

\begin{lemma}\label{pure-local.lem}
  Let $(A,\fm)$ be a Noetherian local ring, and $f:A\rightarrow B$ be a pure
  ring homomorphism.
  Then there exists some maximal ideal $\fn$ of $B$ such that $\fn\cap A=\fm$, and
  $A\rightarrow B_\fn$ is pure.
\end{lemma}

\begin{proof}
  Let $E=E_A(A/\fm)$ be the injective hull of the residue field of $A$.
  Then $j_E(1)\neq 0$ in $B\otimes_A E$ \cite[(6.11)]{HR},
  where $1$ is the image of $1\in A/\fm$ in $E$.
  So there exists some maximal ideal $\fn $ of $B$ such that $j_E(1)$ is still nonzero
  in $B_\fn\otimes_A E$.
  Then $A\rightarrow B_\fn$ is pure by \cite[(6.11)]{HR} again.
  As $\Spec B_\fn\rightarrow \Spec A$ is surjective, we must have $\fn\cap A=\fm$.
\end{proof}

\begin{lemma}\label{dim1.lem}
  If $A$ is a Noetherian ring and $f:A\rightarrow B$ is a pure ring homomorphism,
  then $\dim A\leq \dim B$.
\end{lemma}

\begin{proof}
  Let $\fp_0\supsetneq \fp_1\supsetneq \fp_2\supsetneq \cdots\supsetneq \fp_d$ be
  a chain of prime ideals of $A$.
  Note that $A_{\fp_0}\rightarrow B_{\fp_0}$ is pure by \cite[(6.2)]{HR}.
  By Lemma~\ref{pure-local.lem}, there exists some prime ideal $P_0$ of $B$ lying
  over $\fp_0$ such that $A_{\fp_0}\rightarrow (B_{\fp_0})_{P_0B_{\fp_0}}=B_{P_0}$ is pure.
  Using this argument to the pure ring homomorphism
  $A_{\fp_0}\rightarrow B_{P_0}$, we know that there exists some prime ideal $P_1\subset P_0$
  such that $P_1\cap A=\fp_1$ and $A_{\fp_1}\rightarrow B_{P_1}$ is pure.
  Continuing this, we can take a chain of prime ideals
  \[
  P_0\supsetneq P_1\supsetneq P_2\supsetneq \cdots \supsetneq P_d
  \]
  such that $P_i\cap A=\fp_i$ and $A_{\fp_i}\rightarrow B_{P_i}$ is pure.
  This proves the lemma.
\end{proof}

\begin{proposition}\label{depth.prop}
  Let $f:(A,\fm)\rightarrow (B,\fn)$ be a pure local ring homomorphism such that
  $B$ is Noetherian and the fiber ring $B/\fm B$ is Artinian.
  Then $\dim A=\dim B$ and $\depth A\geq \depth B$.
\end{proposition}

\begin{proof}
  We have $\dim A\leq \dim B\leq \dim A+\dim B/\fm B=\dim A$ by Lemma~\ref{dim1.lem}
  and the assumption that $\dim B/\fm B$ is Artinian.

  Let $r=\depth A$ and $s=\depth B$.
  We prove that $r\geq s$ by induction on $s$.
  If $s=0$, then the assertion is trivial.
  Assume that $s>0$.
  If $r=0$, then there exists some $x\in \fm$ such that $0:_A x=\fm$.
  As $f$ is an injective map, $\fm B\subset 0:_Bx\subset \fn$.
  As $B/\fm B$ is Artinian, $B/(0:_Bx)$ is Artinian and is nonzero.
  Hence
  \[
  \Ass_B B\supset \Ass_B Bx=\Ass_B B/(0:_B x)=\{\fn\},
  \]
  and $s=0$.
  This is a contradiction.
  So $r>0$.
  For each $P\in\Ass B$, we have $P\cap A\neq \fm$, since $\fn \cap A=\fm$, $P\subsetneq \fn$,
  and $B/\fm B$ is zero-dimensional.
  By prime avoidance, there exists some
  \[
  y\in \fm \setminus\left((\bigcup_{P\in\Ass B}(P\cap A))\cup(\bigcup_{\fp\in\Ass A}\fp)\right).
  \]
  So $y$ is $A$-regular $B$-regular.
  By \cite[(6.2)]{HR}, $A/yA\rightarrow B/yB$ is pure, and $(B/yB)/\fm(B/yB)$ is
  Artinian.
  By induction assumption, we have
\[
r-1=\depth A/yA \geq \depth B/yB=s-1
\]
  and hence $r\geq s$, as required.
\end{proof}

\begin{theorem}[Hochster--Eagon \cite{HE}]\label{he.thm}
  Let $f:A\rightarrow B$ be a pure ring homomorphism.
  Let $B$ be Noetherian and Cohen--Macaulay.
  If the fiber ring $B\otimes_A \kappa(\fm)$ is zero-dimensional for any maximal
  ideal $\fm$ of $A$, then $A$ is Cohen--Macaulay.
\end{theorem}

\begin{proof}
  It suffices to show that $A_\fm$ is a Cohen--Macaulay local ring for any maximal
  ideal $\fm$ of $A$.
  Replacing $f$ by $f_\fm:A_\fm\rightarrow B_\fm$, we may assume that $(A,\fm)$ is local.
  Then we can find a maximal ideal $\fn$ of $B$ lying over $\fm$ such that
  $A\rightarrow B_\fn$ is pure.
  As a localization of a zero-dimensional ring is still zero-dimensional, we may
  assume that $f$ is a local homomorphism between local rings.
  By the Cohen--Macaulay property of $B$ and Proposition~\ref{depth.prop},
  we have
  \[
  \depth A\geq \depth B=\dim B=\dim A\geq \depth A,
  \]
  and the assertion follows.
\end{proof}

\begin{theorem}\label{main.thm}
  Let $f:A\rightarrow B$ be a pure ring homomorphism.
  Let $B$ be a Noetherian ring which satisfies Serre's $(S_n)$ condition.
  If the fiber ring $B\otimes_A \kappa(\fp)$ is zero-dimensional for any prime
  ideal $\fp$ of $A$, then $A$ satisfies $(S_n)$, too.
\end{theorem}

\begin{proof}
  Let $\fp$ be a prime ideal of $A$ such that $\depth A_\fp<n$.
  Then we can find a prime ideal $P$ of $B$ lying over $\fp$
  such that the local homomorphism $A_\fp\rightarrow B_P$ is pure.
  Note that the dimension of the fiber ring $\kappa(\fp)\otimes_{A_\fp}B_P$ is
  zero-dimensional.
  Since
  $\depth B_P\leq\depth A_\fp<n$ by Proposition~\ref{depth.prop} and $B$ satisfies the $(S_n)$ condition, we have
  that $B_P$ is Cohen--Macaulay.
  By Theorem~\ref{he.thm}, we have that $A_\fp$ is Cohen--Macaulay.
  This shows that $A$ satisfies the $(S_n)$ condition.
\end{proof}

\begin{corollary}\label{main.cor}
  Let $G$ be a finite group acting on a Noetherian ring $B$.
  Assume that the order $|G|$ of $G$ belongs to $B^\times$, the unit group of $B$.
  If $B$ satisfies Serre's $(S_n)$ condition, then so does the ring of invariants $B^G$.
\end{corollary}

\begin{proof}
  For $b\in B$, the polynomial $f(t)=\prod_{g\in G}(t-gb)$ is monic and lies in $B^G[t]$.
  As $f(b)=0$, the element $b$ is integral over $B^G$.
  This shows that $B^G\hookrightarrow B$ is an integral extension, and all the fiber rings are zero-dimensional.

  Let $R=\ZZ[|G|^{-1}]$, and define the Reynolds operator to be
  $\rho=|G|^{-1}\sum_{g\in G}g\in RG$.
  It is a central idempotent of the group algebra $RG$ of $G$ over $R$.
  As $RG$ acts on $B$, the $G$-linear map $\rho:B\rightarrow B$ is defined.
  It is easy to see that $\rho(B)\subset B^G$ and $\rho$ is the identity map on $B^G$.
  So $\rho(B)=B^G$, and $\rho:B\rightarrow B^G$ is the left inverse of the inclusion
  $j:B^G\rightarrow B$.
  Moreover,
  \[
  \rho(ab)=|G|^{-1}\sum_{g\in G}a(gb)=a(\rho b)
  \]
  for $a\in B^G$ and $b\in B$.
  Namely, $\rho$ is the left inverse of $j$ as a $B^G$-linear map, and thus $B^G$ is a
  direct summand subring of $B$.
  It follows that $B^G$ is a pure subring of $B$.
  The corollary follows from Theorem~\ref{main.thm}.
\end{proof}

\begin{flushleft}
Mitsuyasu HASHIMOTO\\
Department of Mathematics\\
Osaka Metropolitan University\\
Sumiyoshi-ku, Osaka 558--8585, JAPAN\\
e-mail: {\tt mh7@omu.ac.jp}
\end{flushleft}


\begin{thebibliography}{GHKN}
\def\ji#1#2(#3)#4-#5.{\newblock{\em#1} {\bf#2} (#3), #4--#5.}
\def\GTM#1{Graduate Texts in Math. {\bf #1}, Springer}
\def\SLN#1{Lecture Notes in Math. {\bf #1}, Springer}

\bibitem[Bou]{Boutot}
  J.-F. Boutot,
  Singularit\'e rationnelles et quotients par les groupes r\'eductifs,
  {\em Invent. Math.} {\bf 88} (1987), 65--68.


\bibitem[HE]{HE}
  M. Hochster and J. A. Eagon,
  Cohen--Macaulay rings, invariant theory, and the generic perfection of determinantal loci,
  {\em Amer. J. Math.} {\bf 93} (1971), 1020--1058.

\bibitem[HH]{HH}
  M. Hochster and C. Huneke,
  Applications of the existence of big Cohen--Macaulay algebras,
  {\em Adv. Math.} {\bf 113} (1995), 45--117.

\bibitem[HR]{HR}
  M. Hochster and J. Roberts,
  Rings of invariants of reductive groups action on regular rings are Cohen--Macaulay,
  {\em Adv. Math.} {\bf 13} (1974), 115--175.

\bibitem[Kaw]{Kawasaki}
  Takesi Kawasaki,
  On arithmetic Macaulayfication of Noetherian rings,
  {\em Trans. Amer. Math. Soc.} {\bf 354} (2002), 123--149.

\bibitem[Sch]{Schoutens}
H. Schoutens,
Pure subrings of regular rings are pseudo-rational,
{\em Trans. Amer. Math. Soc.} {\bf 360} (2008), 609--627.

\end{thebibliography}
\end{document}